\documentclass[reqno]{amsart}

\usepackage{enumitem}
\usepackage{amssymb,amsmath,mathtools}
\usepackage{hyperref}

\usepackage{microtype}
\usepackage{bbm}

\newcommand*{\mailto}[1]{\href{mailto:#1}{\nolinkurl{#1}}}

\newcommand{\msc}[1]{\href{http://www.ams.org/msc/msc2010.html?t=&s=#1}{#1}}

\newtheorem{theorem}{Theorem}[section]
\newtheorem{lemma}[theorem]{Lemma}
\newtheorem{corollary}[theorem]{Corollary}
\newtheorem{definition}[theorem]{Definition}
\newtheorem{remark}[theorem]{Remark}

\newcommand{\R}{{\mathbb R}}
\newcommand{\N}{{\mathbb N}}
\newcommand{\C}{{\mathbb C}}
\newcommand{\cRH}{{\mathcal R}_{H}}
\newcommand{\id}{{\mathbbm{1}}}
\newcommand{\OO}{{\mathcal O}}
\newcommand{\cW}{{\mathcal W}}
\newcommand{\cR}{{\mathcal R}}
\newcommand{\loc}{\mathrm{loc}}

\newcommand{\I}{\mathrm{i}}
\newcommand{\E}{\mathrm{e}}

\DeclareMathOperator{\re}{Re}
\DeclareMathOperator{\im}{Im}

\newcommand{\nn}{\nonumber}
\newcommand{\be}{\begin{equation}}
\newcommand{\ee}{\end{equation}}

\newcommand{\ti}{\tilde}
\newcommand{\abs}[1]{\left\lvert #1 \right\rvert}
\newcommand{\abss}[1]{\lvert #1 \rvert}
\newcommand{\norm}[1]{\left\lVert #1 \right\rVert}

\newcommand{\inner}[2]{\left\langle#1,#2\right\rangle}

\newcommand{\lam}{\lambda}

\numberwithin{equation}{section}

\newcommand{\hyp}[5]{\,\mbox{}_{#1}F_{#2}\!\left(
  \genfrac{}{}{0pt}{}{#3}{#4};#5\right)}
  
\makeatletter
\newcommand{\dlmf}[1]{%
\cite[%
  \def\nextitem{\def\nextitem{, }}%
  \@for \el:=#1\do{\nextitem\href{http://dlmf.nist.gov/\el}{(\el)}}%
]{dlmf}%
}
\makeatother

\begin{document}

\title[Dispersion Estimates]{Dispersion Estimates for Spherical Schr\"odinger Equations with Critical Angular Momentum}

\dedicatory{To Helge Holden, inspiring colleague and friend, on the occasion of his 60th birthday}

\author[M. Holzleitner]{Markus Holzleitner}
\address{Faculty of Mathematics\\ University of Vienna\\
Oskar-Morgenstern-Platz 1\\ 1090 Wien\\ Austria}
\email{\mailto{amhang1@gmx.at}}

\author[A. Kostenko]{Aleksey Kostenko}
\address{Faculty of Mathematics\\ University of Vienna\\
Oskar-Morgenstern-Platz 1\\ 1090 Wien\\ Austria}
\email{\mailto{duzer80@gmail.com};\mailto{Oleksiy.Kostenko@univie.ac.at}}
\urladdr{\url{http://www.mat.univie.ac.at/~kostenko/}}

\author[G. Teschl]{Gerald Teschl}
\address{Faculty of Mathematics\\ University of Vienna\\
Oskar-Morgenstern-Platz 1\\ 1090 Wien\\ Austria\\ and International 
Erwin Schr\"odinger Institute for Mathematical Physics\\ 
Boltzmanngasse 9\\ 1090 Wien\\ Austria}
\email{\mailto{Gerald.Teschl@univie.ac.at}}
\urladdr{\url{http://www.mat.univie.ac.at/~gerald/}}

\thanks{{\it Research supported by the Austrian Science Fund (FWF) 
under Grants No.\ P26060 and W1245.}}
\thanks{in Partial Differential Equations, Mathematical Physics, and Stochastic Analysis, F.\ Gesztesy et al. (eds), 319--347, EMS Congress Reports {\bf 14}, 2018}


\begin{abstract}
We derive a dispersion estimate for one-dimensional perturbed radial Schr\"o\-din\-ger operators,
where the angular momentum takes the critical value $l=-\frac{1}{2}$.
We also derive several new estimates for solutions of the underlying differential equation and
investigate the behavior of the Jost function near the edge of the continuous spectrum.
\end{abstract}

\maketitle

%%%%%%%%%%%%%%%%%%%%%%%%%%%%%%%%%%%%%%%%%%%%%%%%%%%%%%%%%%%%%%%%%%%%%%%%%%%%%%%
\section{Introduction}
%%%%%%%%%%%%%%%%%%%%%%%%%%%%%%%%%%%%%%%%%%%%%%%%%%%%%%%%%%%%%%%%%%%%%%%%%%%%%%%

The stationary one-dimensional radial Schr\"odinger equation
\begin{equation} \label{Schrl}
  \I \dot \psi(t,x) = H \psi(t,x), \quad 
  H:= - \frac{d^2}{dx^2} + \frac{l(l+1)}{x^2} + q(x),\quad 
  (t,x)\in\R\times \R_+,
\end{equation}
is a well-studied object in quantum mechanics. Starting from the Schr\"odinger equation with a spherically symmetric potential in three dimensions, one obtains
\eqref{Schrl} with $l$ a nonnegative integer. However, other dimensions will lead to different values for $l$ (see e.g.\ \cite[Sect.~17.F]{wdln}). In particular, the half-integer values arise in two dimensions and hence are equally important. Moreover, the
integer case is typically not more difficult than the case $l>-\frac{1}{2}$ but the borderline case $l=-\frac{1}{2}$ usually imposes additional technical problems.
 For example in \cite{ktt} we 
investigated the dispersive properties of the associated radial Schr\"odinger equation, but were not able to cover the case $l=-\frac{1}{2}$. This was
also partly due to the fact that several results we relied upon were only available for the case $l>-\frac{1}{2}$. The
present paper aims at filling this gap by investigating
\begin{equation} \label{Schr}
  \I \dot \psi(t,x) = H \psi(t,x), \quad 
  H:= - \frac{d^2}{dx^2} - \frac{1}{4x^2} + q(x),\quad 
  (t,x)\in\R\times \R_+,
\end{equation}
with real locally integrable potential $q$. We will use $\tau$ to describe the formal Sturm--Liouville differential expression and
$H$ the self-adjoint operator acting in $L^2(\R_+)$ and given by $\tau$ together with the Friedrichs boundary condition at $x=0$:
\be\label{eq:bc}
\lim_{x\to0} W(\sqrt{x},f(x))=0.
\ee

More specifically, our goal is to provide dispersive 
decay estimates for these equations. To this end we
recall that under the assumption
\[
\int_0^\infty x(1+|\log(x)|) |q(x)| dx<\infty
\] 
the operator $H$ has a purely absolutely continuous spectrum
on $[0,\infty)$ plus a finite number of eigenvalues in $(-\infty,0)$  (see, e.g., \cite[Theorem 5.1]{seto} and \cite[Sect.~9.7]{tschroe}). 

Then our main result reads as follows:
%%%%%%%%%%%%%%%%%%%%%%%%%%%%%%%%%%%%%%%%%%%%%%%%%%%%%%%%%%%%%%%%%%%%%%%%%%%%%%%
\begin{theorem}\label{Main}
Assume that 
\be\label{eq:q_hyp}
\int_0^1 |q(x)| dx <\infty \quad \text{and} \quad  
\int_1^\infty x\log^2(1+x) |q(x)|dx<\infty,
\ee
and suppose there is no resonance at $0$ (see Definition \ref{def:resonance}). 
Then the following decay holds
\begin{equation}\label{full}
\norm{\E^{-\I tH}P_c(H)}_{L^1(\R_+)\to L^\infty(\R_+)}
	= \mathcal{O}(|t|^{-1/2}),
\quad t\to\infty.
\end{equation}
Here $P_c(H)$ is the orthogonal projection in $L^2(\R_+)$ onto 
the continuous spectrum of $H$.
\end{theorem}
%%%%%%%%%%%%%%%%%%%%%%%%%%%%%%%%%%%%%%%%%%%%%%%%%%%%%%%%%%%%%%%%%%%%%%%%%%%%%%%

Such dispersive estimates for Schr\"odinger equations have a long tradition and here we refer to a brief selection of articles \cite{bpstz1,bpstz2,EKMT,GS,HKT,K10,ktt,kotr,schlag,wed,wed2}, where further references can be found.
We will show this result by establishing a corresponding low energy result, Theorem~\ref{thm:le2} (see also Theorem~\ref{thm:le1}), and a corresponding high energy result, Theorem~\ref{thm:he}. Our proof is based on the approach proposed in \cite{ktt}, however, the main technical difficulty is the analysis of the low and high energy behavior of the corresponding Jost function. Let us also mention that the potential $q\equiv0$ does not satisfy the conditions of Theorem \ref{Main}, that is, there is a resonance at $0$ in this case. However, it is known that the dispersive decay \eqref{full} holds true if $q\equiv 0$ \cite{kt} and hence Theorem \ref{Main} states that the corresponding estimate remains true under additive non-resonant perturbations.
For related results on scattering theory for such operators we refer to \cite{bg1,bg2}.

Finally, let us briefly describe the content of the paper. Section \ref{wa-sec} is of preliminary character, where we collect and derive some necessary estimates for solutions, the Green's function and the high and low energy behavior of the Jost function \eqref{eq:JostFunct}. However, we would like to emphasize that the behavior of the Jost function near the bottom of the essential spectrum is still not understood
satisfactorily, and for this very reason the resonant case had to be excluded from our main theorem. The proof of Theorem \ref{Main} is given in Section \ref{ll-sec}. In order to make the exposition self-contained, we gathered the appropriate version of the van der Corput lemma and necessary facts on the Wiener algebra in Appendix \ref{sec:vdCorput}. Appendix \ref{sec:Bessel} contains relevant facts about Bessel and Hankel functions. 

%%%%%%%%%%%%%%%%%%%%%%%%%%%%%%%%%%%%%%%%%%%%%%%%%%%%%%%%%%%%%%%%%%%%%%%%%
\section{Properties of solutions}
\label{wa-sec}
%%%%%%%%%%%%%%%%%%%%%%%%%%%%%%%%%%%%%%%%%%%%%%%%%%%%%%%%%%%%%%%%%%%%%%%%%

In this section we will collect some properties of the solutions of the underlying
differential equation required for our main results.

\subsection{The regular solution}

Suppose that
\be\label{q:hyp}
q\in L^1_\loc(\R_+)\quad \text{and} \quad \int_0^1 x\big(1-\log(x) \big)|q(x)|dx <\infty.
\ee
Then the ordinary differential equation 
\[
\tau f= z f,\qquad \tau := - \frac{d^2}{dx^2} - \frac{1}{4x^2} + q(x),
\]
has a system of solutions $\phi(z,x)$ and $\theta(z,x)$ which are real 
entire with respect to $z$ and such that 
\be\label{eq:fs01}
\phi(z,x) 
	= \sqrt{\frac{\pi x}{2}} \ti{\phi}(z,x),\quad  
\theta(z,x)	=-\sqrt{\frac{2x}{\pi}}\log(x)\ti{\theta}(z,x),
\ee
where $\ti{\phi}(z,\cdot)\in W^{1,1}[0,1]$,  $\ti{\theta}(z,\cdot)\in C[0,1]$ and 
$\ti{\phi}(z,0)=\ti{\theta}(z,0)=1$. Moreover, we can choose $\theta(z,x)$ such that $\lim_{x\to 0}W(\sqrt{x}\log(x),\theta(z,x))=0$ for all $z\in\C$.  Here $W(u,v) = u(x)v'(x) - u'(x)v(x)$ is the usual Wronski determinant.
For a detailed construction of these 
solutions we refer to, e.g., \cite{kt}.
 
We start with two lemmas containing estimates for the Green's function 
of the unperturbed equation 
\[ 
G_{-\frac{1}{2}}(z,x,y) = \phi_{-\frac{1}{2}}(z,x) \theta_{-\frac{1}{2}}(z,y) - \phi_{-\frac{1}{2}}(z,y) \theta_{-\frac{1}{2}}(z,x)
\]
and the regular solution $\phi(z,x)$ 
(see, e.g., \cite[Lemmas~2.2, A.1, and A.2]{kst}).  Here 
\be\label{eq:phi0}
\phi_{-\frac{1}{2}}(z,x) 
	=  \sqrt{\frac{\pi x}{2}} 
	  J_{0}(\sqrt{z} x),
\quad 
\theta_{-\frac{1}{2}}(z,x) 
= \sqrt{\frac{\pi x}{2}} \left(
		\frac{1}{\pi}\log(z)J_{0}(\sqrt{z} x)
	    -Y_{0}(\sqrt{z} x)\right), 
	    \ee
where $J_{0}$ and $Y_{0}$ are the usual Bessel and Neumann 
functions (see Appendix \ref{sec:Bessel}). All branch
cuts are chosen along the negative real axis unless explicitly stated
otherwise.

The first two results are essentially from \cite[Appendix~A]{kst}. However, since the focus there was on a finite
interval, some small adaptions are necessary to cover the present case of a half-line.

\begin{lemma}[\cite{kst}]\label{lem:b.1}
The following estimates hold:
\begin{align}\label{estphil}
\abs{\phi_{-\frac{1}{2}}(k^2,x)} 
	&\le C \left(\frac{x}{1+\abs{k}x}\right)^{\frac12} \E^{\abs{\im\, k} x},\\
\abs{\theta_{-\frac{1}{2}}(k^2,x)} 
	&\le C \left(\frac{x}{1+\abs{k}x}\right)^{\frac12} \left(1+\abs{\log\left(\frac{1+\abs{k}x}{x}\right)}\right) \E^{\abs{\im\, k} x}, \label{estthetal}
\end{align}
for all $x>0$, and
\begin{align}\label{estGl}
\abs{G_{-\frac{1}{2}}(k^2,x,y)} \leq C \left(\frac{x}{1+ |k| x}\right)^{\frac{1}{2}}\left(\frac{y}{1+ |k| y}\right)^{\frac{1}{2}}
\left( 1+ \log\Big(\frac{x}{y}\Big)\right)\E^{ \abs{\im\, k} (x-y)}
\end{align}
for all $0< y \leq x<\infty$. 
\end{lemma}
\begin{proof}
The first two estimates are clear from the asymptotic behavior of the Bessel function $J_0$ and the Neumann function $Y_0$ (see \eqref{eq:Jnu01}, \eqref{eq:Ynu01} and \eqref{eq:asymp-J-infty},  \eqref{eq:asymp-Y-infty}).

To consider the third one, first of all we have 
\begin{align}\label{Gdef}
\begin{split}
G_{-\frac12}(k^2,x,y) 
&= - \frac{\pi}{2}\sqrt{xy} \big[J_{0}(kx)Y_{0}(ky) 
   - J_{0}(ky)Y_{0}(kx)\big]
\\[1mm]
&= - \frac{\I\pi}{4}\sqrt{xy} \left[ H^{(1)}_{0}(kx)H^{(2)}_{0}(ky)    
   - H^{(1)}_{0}(ky)H^{(2)}_{0}(kx)\right].
   \end{split}
\end{align}

We divide the proof of \eqref{estGl} in three steps.

{\em Step (i):   $\abs{ky} \le \abs{kx} \le 1$.} Using the first equality in \eqref{Gdef} and employing \eqref{eq:Jnu01} and \eqref{eq:Ynu01}, we get
\[
\abs{G_{-\frac12}(k^2,x,y)} \le C \sqrt{xy} \left(1+\log\Big(\frac{|k|x}{|k|y}\Big) \right) = C \sqrt{xy} \left(1+\log\Big(\frac{x}{y}\Big) \right),
\]
which immediately implies \eqref{estGl}.

{\em Step (ii):  $\abs{ky}\le 1\le \abs{kx}$}. 
Using the asymptotics \eqref{eq:Jnu01}--\eqref{eq:asymp-Y-infty} from Appendix B, we get
\begin{align*}
\abs{G_{-\frac12}(k^2,x,y)}\le C \sqrt{xy} \sqrt{\frac{1}{|k|x}}\E^{ \abs{\im\, k} (x-y)}\left(1- \log(|k|y) \right).
\end{align*}
We arrive at \eqref{estGl} by noting that
\[
0<-\log(|k|y)\le \log(x/y)
\]
since $|k|y\le 1 \le |k|x$.

{\em Step (iii): $1\le\abs{ky}\le\abs{kx}$.} For the remaining case it suffices to use the
second equality in \eqref{Gdef} and \eqref{eq:asymp-H1-infty}--\eqref{eq:asymp-H2-infty} to arrive at
\[
 \abs{G_{-\frac12}(k^2,x,y)}\le C \sqrt{xy} \sqrt{\frac{1}{|k|x|k|y}}  \E^{ \abs{\im\, k} (x-y)} = \frac{C}{|k|}  \E^{ \abs{\im\, k} (x-y)} ,
\]
which implies the claim.
\end{proof}

\begin{lemma}[\cite{kst}]
\label{lem:phi}
Assume \eqref{q:hyp}. Then $\phi(z,x)$ satisfies the integral equation
\be \label{eq:phi_int}
\phi(z,x) = \phi_{-\frac12}(z,x) + \int_0^x G_{-\frac12}(z,x,y)  \phi(z,y) q(y) dy.
\ee
Moreover, $\phi(\cdot,x)$ is entire for every $x>0$ and satisfies the estimate
\begin{align}
\abs{\phi(k^2,x) - \phi_{-\frac12}(k^2,x)} 
	\leq &C \left(\frac{x}{1+ \abs{k} x}\right)^{\frac12} 
	\E^{\abs{\im\, k} x}  \nn\\
&\times \int_0^x \frac{y}{1+|k|y} \left(1+ \log\Big(\frac{x}{y}\Big)\right)|q(y)|dy\label{estphi}
\end{align}
for all $x>0$ and $k\in\C$.
\end{lemma}
\begin{proof}
The proof is based on the successive iteration procedure. As in the proof of Lemma 2.2 in \cite{kst}, set
\[
\phi = \sum_{n=0}^\infty \phi_{n},\quad \phi_0=\phi_{-\frac12},\quad \phi_n(k^2,x):=\int_0^x G_{-\frac12}(k^2,x,y)\phi_{n-1}(k^2,y)q(y)dy
\]
for all $n\in\N$. The series is absolutely convergent since 
\be\label{eq:phi_n}
\begin{split}
\abs{{\phi}_n(k^2,x)}
\le \frac{C^{n+1}}{n!}&
	\left(\frac{ x}{1+|k|x}\right)^{\frac{1}{2}}
	\E^{\abss{\im k} x} \\
	&\times \left(\int_0^x \frac{y}{1+|k|y} \left(1+ \log\Big(\frac{x}{y}\Big)\right)|q(y)|dy\right)^n,\quad n\in\N.
	\end{split}
\ee
This is all we need to finish the proof of this lemma. 
\end{proof}

We also need the estimates for derivatives.

\begin{lemma}
\label{lem:part-z}
The following estimates
hold
\be\label{eq:partial-z-phil}
\abss{\partial_{k}{\phi}_{-\frac12}(k^2,x)}
\le C |k|x 
	\left(\frac{ x}{1+\abs{k}x}\right)^{\frac32}
	\E^{\abss{\im k} x}
\ee
for all $x>0$, and 
\begin{align} \label{eq:partial-z-Gl}
\begin{split}
\abs{\partial_k G_{-\frac12}(k^2,x,y)} \le  C |k|x \left(\frac{x}{1+ |k| x}\right)^{\frac{3}{2}} &\left( \frac{y}{1+ |k| y}\right)^{\frac{1}{2}}\\
&\times\left( 1+ \log\Big(\frac{x}{y}\Big)\right)\E^{ \abs{\im\, k} (x-y)},   
\end{split}
\end{align}
for all $0<y \le x<\infty$.
\end{lemma}
\begin{proof}
The first inequality follows from the identity (see \dlmf{10.6.3})
\[
\partial_k{\phi}_{-\frac12}(k^2,x) = -x\sqrt{\frac{\pi x}{2}} J_1(kx)
\]
along with the asymptotic behavior of the Bessel function $J_1$ (cf. \cite[Lemma 2.1]{ktt}).

To prove \eqref{eq:partial-z-Gl}, we first calculate 
\begin{align}\label{eq:partial-z-Gl-proof}
\begin{split}
\partial_k G_{-\frac12}(k^2,x,y) 
=&\; \frac{\pi}{2}\sqrt{xy} \Big[xJ_1(kx)Y_0(ky) 
     - y J_1(ky)Y_0(kx)
\\
 &\qquad\qquad - 
     x J_0(ky)Y_{1}(kx)+ yJ_0(kx)Y_{1}(ky)\Big]
\\
=&\; \frac{\I\pi}{4}\sqrt{xy} \left[xH^{(1)}_1(kx)H^{(2)}_0(ky)
     - y H^{(1)}_1(ky)H^{(2)}_0(kx)\right.
\\
 &\qquad\qquad \left. +x H^{(1)}_0(ky)H^{(2)}_{1}(kx) - y H^{(1)}_0(kx)H^{(2)}_{1}(ky) \right],
\end{split}
\end{align}
where we have used formulas \eqref{Gdef} and the identities for derivatives of Bessel and Hankel functions (cf. Appendix B). 

{\em Step (i):   $\abs{ky}\le\abs{kx}\le 1$.} 
Employing the series expansions \eqref{eq:Jnu01}--\eqref{eq:Ynu01} we get from the first equality in \eqref{eq:partial-z-Gl-proof}
\begin{align*}
\partial_k G_{-\frac12}(k^2,x,y)&=\frac{\pi}{2}\sqrt{xy} \left[x\frac{kx}{4}\frac{2\log(ky)}{\pi}-y\frac{ky}{4}\frac{2\log (kx)}{\pi}\right. \\
&-x\Big(\frac{1}{2\pi kx}+\frac{2\log(kx)}{\pi}\frac{kx}{4}\Big)+y\Big(\frac{1}{2\pi ky}+\frac{2\log(ky) }{\pi}\frac{ky}{4}\Big) \Big] (1+\OO(1)) \\
&=\frac{\pi}{2}\sqrt{xy}\big(kx^2+ky^2\big)\big(\log(ky)-\log(kx)\big) (1+\OO(1))\\
&=\frac{\pi}{2}\sqrt{xy}kx^2\log(y/x) (1+\OO(1)).
\end{align*}
This immediately implies the desired claim. 

{\em Step (ii):  $\abs{ky}\le 1\le\abs{kx}$.} Again we employ the asymptotics \eqref{eq:Jnu01}--\eqref{eq:asymp-Y-infty} from Appendix B to get:
\begin{align*}
\partial_k G_{-\frac12}(k^2,x,y)&=\frac{\pi\sqrt{xy}}{2} \left[  \sqrt{\frac{2x}{\pi k}}\cos\Big(kx-\frac{3 \pi}{ 4}\Big)\frac{2\log(ky)}{\pi}-yky \sqrt{\frac{2}{\pi kx}}\cos\Big(kx-\frac{\pi}{ 4}\Big) \right.\\
&\left.\ \ -\sqrt{\frac{2x}{\pi k}}\cos\Big(kx-\frac{3 \pi}{ 4}\Big)+y\sqrt{\frac{2}{\pi kx}}\cos\Big(kx-\frac \pi 4\Big) \frac{1}{2\pi ky}\right](1+\OO(1))\\
&=\frac{\pi\sqrt{xy}}{2} \left[\sqrt{\frac{2x}{\pi k}}\cos\Big(kx-\frac{3 \pi}{ 4}\Big)\Big(\frac{2}{\pi}\log(ky) - 1\Big)  \right.\\
&\qquad \qquad\left. + \sqrt{\frac{2}{\pi kx}}\cos\Big(kx-\frac \pi 4\Big)\Big( \frac{1}{2\pi k} - yky\Big)\right](1+\OO(1)).
\end{align*}
This gives the desired estimate, where we have to use $\frac{1}{|k|} \le x$ to estimate the second summand and the logarithmic expression appropriately (cf. step (ii) of \ref{lem:b.1}).

{\em Step (iii): $1\le\abs{ky}\le\abs{kx}$.} To deal with the remaining case we shall use the second equality in \eqref{eq:partial-z-Gl-proof} and the asymptotic 
expansions of Hankel functions \eqref{eq:asymp-H1-infty}--\eqref{eq:asymp-H2-infty}: 
\begin{align*}
\partial_k G_{-\frac12}(k^2,x,y) &=\frac{\I\pi\sqrt{xy}}{4} \left[x\frac{2}{\pi k\sqrt{xy}}\E^{\I k(x-y) -\I\pi/2} - y\frac{2}{\pi k\sqrt{xy}}\E^{\I k(y-x) -\I\pi/2} \right.\\
&\quad +\left. x\frac{2}{\pi k\sqrt{xy}}\E^{\I k(y-x) + \I\pi/2} - y\frac{2}{\pi k\sqrt{xy}}\E^{\I k(x-y) + \I\pi/2}\right](1+\OO(1))\\
&= \frac{x+y}{2\I k}\sin(k(x-y)) 
(1+\OO(1)).
\end{align*}
This again immediately implies \eqref{eq:partial-z-Gl}.
\end{proof}

\begin{lemma}
\label{lem:about-partial-z}
Assume \eqref{q:hyp}. Then  
$\partial_k{\phi}(k^2,x)$ is a solution to the integral equation
\begin{multline}
\label{eq:partial-z-phi}
\partial_k {\phi}(k^2,x)
= \partial_k{\phi}_{-\frac12}(k^2,x)\\
 	+ \int_0^x [\partial_k G_{-\frac12}(k^2,x,y)] {\phi}(k^2,y) 
 	+ G_{-\frac12}(k^2,x,y) \partial_k{\phi}(k^2,y)]q(y)dy 
\end{multline}
and satisfies the estimate
\begin{align}
\abs{\partial_k{\phi}(k^2,x) -\partial_k {\phi}_{-\frac12}(k^2,x)}
\le& C|k|x \left(\frac{x}{1+ \abs{k} x}\right)^{\frac32} \E^{\abs{\im\, k} x}  \label{eq:partial-z-diff-phi}\\
&\times \int_0^x \frac{y}{1+|k|y} \left(1+ \log\Big(\frac{x}{y}\Big)\right)|q(y)|dy.\nn
\end{align}
\end{lemma}

\begin{proof}
Let us show that $\partial_k \phi(k^2,x)$ given by
\begin{gather}
\partial_k \phi = \sum_{n=0}^\infty \beta_n,\quad  \beta_0(k,x)
	=\partial_k{\phi}_{-\frac12}(k^2,x),\label{eq:betaser}
	\\
\begin{multlined}	
\beta_{n}(k,x)
	=\int_0^x \partial_k G_{-\frac12}(k^2,x,y)\,  {\phi}_{n-1}(k^2,y) q(y)dy
		\qquad\qquad\qquad\qquad
		\\
		+ \int_0^x G_{-\frac12}(k^2,x,y) \beta_{n-1}(k,y) q(y)dy,
		\quad n\in \N,\label{eq:partial-z-first}
\end{multlined}		 
\end{gather}
satisfies \eqref{eq:partial-z-phi}. Here $\phi_n$ is defined in Lemma \ref{lem:phi}. Using \eqref{eq:phi_n} and 
\eqref{eq:partial-z-phil}, we can bound the first summand in 
\eqref{eq:partial-z-first} as follows
\begin{align*}
\abs{\text{1st term}}
&\le
\frac{C^{n+1}}{(n-1)!} |k|x \left(\frac{x}{1+|k|x}\right)^{\frac{3}{2}}\! \E^{\abss{\im k} x} \\
&\int_0^x \left(1+ \log\Big(\frac{x}{y}\Big)\right) \frac{y\abs{q(y)}}{1+|k|y}\left(\int_0^y 
\left(1+ \log\Big(\frac{y}{t}\Big) \right)\frac{t\abs{q(t)}}{1+|k|t}dt\right)^{n-1}\!dy\\
&\le \frac{C^{n+1}}{n!}|k| x \left(\frac{x}{1+|k|x}\right)^{\frac32} \E^{\abss{\im k} x}
	\left(\int_0^x  \left(1+ \log\Big(\frac{x}{y}\Big)\right) \frac{y |q(y)|}{1+|k|y}dy\right)^{n}.
\end{align*}
Next, using induction, one can show that the second summand admits a 
similar bound and hence we finally get
\[
\abs{\beta_n(k,x)}
\le \frac{C^{n+1}}{n!} |k|x \left(\frac{x}{1+|k|x}\right)^{\frac32} \E^{\abss{\im k} x}
	\left(\int_0^x  \left(1+ \log\Big(\frac{x}{y}\Big)\right) \frac{y |q(y)|}{1+|k|y}dy\right)^{n}.
\] 
This immediately implies the convergence of 
\eqref{eq:betaser} and, moreover, the estimate
\[
\abss{\partial_k{\phi}(k^2,x) -\partial_k{\phi}_{-\frac12}(k^2,x)}
\le \sum_{n=1}^\infty\abs{\beta_n(k,x)},
\]
from which \eqref{eq:partial-z-diff-phi} follows under the assumption
\eqref{q:hyp}.
\end{proof}

Furthermore, by \cite{fad, cc, volk} (see also \cite{HKT2}), the regular solution $\phi$ admits a representation by means of transformation operators preserving the behavior of solutions at $x=0$
(see also \cite[Chap. III]{cs} for further details and historical remarks).

\begin{lemma}\label{lem:toGL}
Suppose  $q\in L^1_{\loc}([0,\infty))$. Then
\be\label{eq:to_GL}
\phi(z,x) = \phi_{-\frac12}(z,x) + \int_0^x B(x,y) \phi_{-\frac12}(z,y) dy =(I+B)\phi_{-\frac12}(z,x),
\ee
where the so-called Gelfand--Levitan kernel $B\colon\R_+^2\to \R$ satisfies the estimate
\be\label{eq:GLest}
|B(x,y)| \le \frac{1}{2} {\sigma}_0\left(\frac{x+y}{2}\right)\E^{{\sigma}_1(x)},\quad {\sigma}_j(x)=\int_0^x s^j |q(s)|ds,
\ee
for all $0<y<x$ and $j\in \{0,1\}$.
\end{lemma}

In particular, this lemma immediately implies the following useful result.

\begin{corollary}\label{cor:Best}
Suppose $q\in L^1((0,1))$. Then $B$ is a bounded operator on $L^\infty((0,1))$.
\end{corollary}

\begin{proof}
If $f\in L^\infty((0,1))$, then using the estimate \eqref{eq:GLest} we get
\begin{multline*}
|(B f)(x)| 
= \Big|\int_0^x B(x,y) f(y) dy\Big| \le \|f\|_\infty \int_0^x |B(x,y)|dy 
\\
\le \frac{1}{2}\|f\|_\infty \E^{{\sigma}_1(1)} \int_0^x {\sigma}_0\Big(\frac{x+y}{2}\Big)dy 
\le \frac{1}{2}  \|f\|_\infty \E^{{\sigma}_1(1)} \sigma_0(1),
\end{multline*}
which proves the claim.
\end{proof}

\begin{remark}\label{rem:Best}
Note that $B$ is a bounded operator on $L^2((0,a))$ for all $a>0$. However, the estimate \eqref{eq:GLest} allows to show that its
norm behaves like $\OO(a)$ as $a\to \infty$ and hence $B$ might not be bounded on $L^2(\R_+)$. 
\end{remark}

\subsection{The Jost solution and the Jost function}
\label{sec:jsol}

In this subsection, we assume that the potential $q$ belongs to 
{\em the Marchenko class}, i.e., in addition to \eqref{q:hyp}, $q$ 
also satisfies 
\be\label{eq:q_mar}
\int_1^\infty x\log(1+x) |q(x)|dx<\infty.
\ee
Recall that under these assumptions on $q$ the spectrum of $H$ is 
purely absolutely continuous on $(0,\infty)$ with an at most finite 
number of eigenvalues $\lam_n \in (-\infty,0)$. 
A solution $f(k,\cdot)$ to $\tau y = k^2y$ with $k\neq 0$ satisfying the following asymptotic normalization 
\be\label{eq:JostSol}
f(k,x) = \E^{\I k x}(1 + o(1)),\qquad f'(k,x) = \I k \E^{\I k x}(1 + o(1))
\ee
as $x\to \infty$,  is called {\em the Jost solution}. 
In the case $q\equiv0$, we have (cf. \eqref{eq:asymp-H1-infty})
\be\label{eq:jost0}
f_{-\frac12}(k,x)
	= \E^{\I\frac{\pi }{4}}\sqrt{\frac{\pi xk}{2}}
	H_0^{(1)}(kx),
\ee
which is analytic in $\C_+$ and continuous in $\overline{\C_+}\setminus\{0\}$.  Here $H^{(1)}_{\nu}$ is the Hankel function
of the first kind (see Appendix \ref{sec:Bessel}). Using the estimates for Hankel functions we obtain
\be
\begin{split}
\label{est:Jostsol}
\abs{f_{-\frac12}(k,x)}
	&\le C\left(\frac{\abs{k}x}{1+ \abs{k} x}\right)^{\frac12} \E^{-\abs{\im\, k} x}\left( 1-\log\left(\frac{|k|x}{1+|k|x}\right) \right)\le C 
		\E^{-\abs{\im\, k} x}
	\end{split}
\ee
 for all $x>0$. Notice that for the second inequality in \eqref{est:Jostsol} we have to use the fact that the function $x\mapsto \sqrt{\frac{x}{x+1}}\log\left( \frac{x}{x+1} \right)$ is bounded on $\R_+$.

\begin{lemma}\label{lem:b.4}
Assume \eqref{eq:q_mar}. Then the Jost solution satisfies the integral equation
\be \label{eq:jost_inteq}
f(k,x) 
	= f_{-\frac12}(k,x) - \int_x^\infty G_{-\frac12}(k^2,x,y) f(k,y) q(y) dy.
\ee
For all $x>0$, $f(\cdot,x)$ is analytic in the upper half plane and can 
be continuously extended to the real axis away from $k=0$ and
\begin{align}
\label{estpsi}
\abss{f(k,x) - f_{-\frac12}(k,x)}
	&\leq C \left(\frac{ x}{1+ \abs{k} x}\right)^{\frac12} 
	\E^{-\abs{\im\, k}\, x} \\
	&\qquad \times \int_x^\infty 
	\left( \frac{y }{1 +\abs{k} y}\right)^{\frac12 }\left(1+ \log\Big(\frac{y}{x}\Big)\right) \abss{q(y)} dy. \nonumber
\end{align}
\end{lemma}

\begin{proof}
The proof is based on the successive iteration procedure. Set
\[
f = \sum_{n=0}^\infty f_{n},\quad 
f_0 = f_{-\frac12},\quad 
f_n(k,x) = -\int_x^\infty G_{-\frac12}(k^2,x,y)f_{n-1}(k,y)q(y)dy
\]
for all $n\in\N$. The series is absolutely convergent since 
\begin{align*}
\abs{{f}_n(k,x)}
\le \frac{C^{n+1}}{n!}&
	\left(\frac{x}{1+|k|x}\right)^{\frac12}
	\E^{-\abss{\im k} x}\\
	& \times \left(\int_x^\infty \left( \frac{y }{1 +\abs{k} y}\right)^{\frac12 } \left(1+ \log\Big(\frac{y}{x}\Big)\right)\abss{q(y)}dy\right)^n
\end{align*}
holds for all $n\in\N$. 
The latter also proves \eqref{estpsi}.
\end{proof}

Furthermore, by \cite{fad, cc, soh, soh2} (see also \cite{HKT2}), the Jost solution $f$ admits a representation by means of transformation operators preserving the behavior of solutions at infinity. 

\begin{lemma}[\cite{soh,soh2}]\label{lem:to}
Assume \eqref{eq:q_mar}  and let $k\neq 0$.
Then 
\be\label{eq:to_Mar}
f(k,x) = f_{-\frac12}(k,x) + \int_x^\infty K(x,y) f_{-\frac12}(k,y) dy = (I+K)f_{-\frac12}(k,x),
\ee
where the so-called Marchenko kernel $K\colon \R^2\to \R$ satisfies the estimate
\be\label{eq:MAest}
|K(x,y)| \le \frac{c_{0}}{2} \ti\sigma_{0}\left(\frac{x+y}{2}\right)\E^{c_{0}\ti{\sigma}_1(x)-\ti{\sigma}_1(\frac{x+y}{2})},\quad \ti{\sigma}_j(x)=\int_x^\infty s^j |q(s)|ds,
\ee
for all $x<y<\infty$. Here $c_{0}$ is a positive constant given by
\[
c_{0} := \sup_{s\in(0,1)} (1-s)^{1/2}\hyp21{1/2,1/2}{1}{s} = \sup_{s\in(0,1)} (1-s)^{1/2}\sum_{n=0}^\infty \frac{((1/2)_n)^2}{(n!)^2}s^n.
\]
\end{lemma}

Notice that $c_0$ is finite in view of \dlmf{15.4.21}. Moreover, this lemma immediately implies the following useful result.

\begin{corollary}\label{cor:Kest}
If \eqref{eq:q_mar} holds, then $K$ is a bounded operator on $L^\infty((1,\infty))$.
\end{corollary}

\begin{proof}
If $f\in L^\infty((1,\infty))$, then using the estimate \eqref{eq:MAest} we get
\begin{align*}
|(K f)(x)| =& \Big|\int_x^\infty K(x,y) f(y) dy\Big| \le \|f\|_\infty \int_x^\infty |K(x,y)|dy \\
&\le \frac{c_0}{2}\|f\|_\infty \E^{c_0\ti{\sigma}_1(x)} \int_1^\infty \ti{\sigma}_{0}\Big(\frac{1+y}{2}\Big)dy \\
&\le c_0 \|f\|_\infty \E^{c_0\ti{\sigma}_1(1)} \int_1^\infty \ti{\sigma}_{0}(s)ds= c_0\|f\|_\infty \big(\ti{\sigma}_{1}(1) - \ti{\sigma}_{0}(1)\big) \E^{c_0\ti{\sigma}_1(1)},
\end{align*}
which proves the claim.
\end{proof}

By Lemma \ref{lem:b.4}, the Jost solution is analytic in the upper half plane and can 
be continuously extended to the real axis away from $k=0$. We can 
extend it to the lower half plane by setting 
$f(k,x) =f(-k,x)= f(k^*,x)^*$ for $\im(k)<0$ (here and below we denote the complex conjugate of $z$ by $z^*$). For $k\in\R\setminus\{0\}$ we obtain two solutions $f(k,x)$ and 
$f(-k,x)=f(k,x)^*$ of the same equation whose Wronskian is given by (cf. \eqref{eq:JostSol})
\be\label{eq:wrfkpm}
W(f(-k,.),f(k,.))= 2\I k.
\ee
{\em The Jost function} is defined as
\be\label{eq:JostFunct}
f(k) := W(f(k,.),\phi(k^2,.))
\ee
and we also set
\[
g(k) := W(f(k,.),\theta(k^2,.))
\]
such that
\be
\label{eq:5.7}
f(k,x) = f(k) \theta(k^2,x) - g(k) \phi(k^2,x).
\ee
In particular, the function given by
\[
m(k^2) := -\frac{g(k)}{f(k)},\quad k\in\C_+,
\]
is called {\em the Weyl $m$-function} (we refer to \cite{kst2,kt2} for further details). 
Note that both $f(k)$ and $g(k)$ are analytic in the upper half plane 
and $f(k)$ has simple zeros at $\I \kappa_n = \sqrt{\lam_n}\in\C_+$. 

Since $f(k,x)^*=f(-k,x)$ for $k\in\R\setminus\{0\}$, we obtain 
$f(k)^*=f(-k)$ and $g(k)^*=g(-k)$. Moreover, \eqref{eq:wrfkpm} shows
\be\label{eq:phif}
\phi(k^2,x) 
	= \frac{f(-k)}{2\I k} f(k,x) - \frac{f(k)}{2\I k} f(-k,x), 
	  \quad k\in\R\setminus\{0\},
\ee
and by \eqref{eq:5.7} we get
\[
2\I \im(f(k) g(k)^*)
	= f(k)g(k)^* -  f(k)^* g(k) 
	= W(f(-k,\cdot),f(k,\cdot))= 2\I k.
\]
Moreover,
\be\label{eq:imm=f}
\im\,m(k^2) = - \frac{\im\big(f(k)^*g(k)\big)}{|f(k)|^2} 
			= \frac{k}{\abs{f(k)}^2}, \quad k\in\R\setminus\{0\}.
\ee
Note that 
\[
f_{-\frac12}(k)= W(f_{-\frac12}(k,.),\phi_{-\frac12}(k^2,.)) = \sqrt{k}\E^{-\I\frac{\pi }{4}}, \quad 0\le \arg(k) < \pi.
\]
Thus, by \cite[Theorem 2.1]{kt2} (see also Eq. (5.15) in \cite{kt2} or \cite{ka56}), on the real line we have
\be\label{eq:JFasymp}
\abs{f(k)}=\sqrt{\abs{k}}(1+o(1)),\quad k\to \infty.
\ee

\subsection{High and low energy behavior of the Jost function}
Consider the following function
\be\label{eq:a.F}
F(k) = \frac{f(k)}{f_{-\frac12}(k)} = \E^{\I\frac{\pi }{4}}\, k^{-\frac{1}{2}}f(k) = \E^{\I\frac{\pi }{4}}\, k^{-\frac12} W(f(k,.),\phi(k^2,.)),\quad 
	  \im\, k\ge 0. 
\ee
Let us summarize the basic properties of $F$.

\begin{lemma}\label{lem:Fprop}
The function $F$ defined by \eqref{eq:a.F} is analytic in $\C_+$ and continuous in $\overline{\C_+}\setminus\{0\}$. Moreover, $F(k)^\ast = F(-k)\neq 0$ for all $k\in\R\setminus\{0\}$ and 
\be\label{eq:Fatinfty}
|F(k)| = 1+o(1)
\ee  
as $k\in\R$ tends to $\infty$.
\end{lemma}

\begin{proof}
The first claim follows from the corresponding properties of the Jost function. Next, \eqref{eq:phif} implies that $f(k)\neq 0$ for all $k\in\R\setminus\{0\}$. Finally,  \eqref{eq:Fatinfty} follows from \eqref{eq:JFasymp}.
\end{proof}

The analysis of the behavior of $F$ near zero is much more delicate. We start with the following integral representation. 

\begin{lemma}[\cite{kt2}]\label{lem:b.5}
Assume \eqref{q:hyp} and \eqref{eq:q_mar}. Then the function $F$ admits 
the integral representation
\begin{align}
F(k) = 1+&\E^{\I \frac{\pi}{4}}k^{-\frac{1}{2}}\int_0^\infty f_{-\frac12}(k,x)\phi(k^2,x) q(x)dx \label{eq:intr_F}
	 \\	 
     &= 1+ \E^{\I\frac{\pi }{4}}\, k^{-\frac12} 
     	\int_0^\infty f(k,x)\phi_{-\frac12}(k^2,x) q(x)dx\nn
\end{align}
for all $k\in \overline{\C_+}\setminus\{0\}$. 
 \end{lemma}

\begin{proof}
To prove the integral representations \eqref{eq:intr_F}, we need to replace $\phi$ and $f$ in \eqref{eq:a.F} by \eqref{eq:phi_int} and \eqref{eq:jost_inteq}, respectively, use the asymptotic estimates for $\phi$, $f$ and $G_{-\frac12}$, 
and then take the limits $x\to +\infty$ and $x\to 0$. 
\end{proof}

\begin{corollary}\label{cor:Fintrep}
Assume in addition that $q$ satisfies  
\be\label{eq:q_mar2}
\int_1^\infty x\log^2(1+x) |q(x)|dx<\infty.
\ee
Then for $k>0$ the integral representation \eqref{eq:intr_F} can be rewritten as follows
\begin{align}\label{eq:F_rep2}
\begin{split}
F(k)=1+&\int_0^{\infty} \theta_{-\frac12}(k^2,x)\phi(k^2,x) q(x)dx \\
& + \Big(\I - \frac{1}{\pi} \log(k^2)\Big) \int_0^{\infty} \phi_{-\frac12}(k^2,x)\phi(k^2,x) q(x)dx.
\end{split}
\end{align} 
\end{corollary}

\begin{proof}
Indeed, the integrals converge for all $k\in\R\setminus\{0\}$ due to \eqref{estphil}, \eqref{estthetal} and \eqref{estphi}. Then \eqref{eq:F_rep2} follows from the first formula in \eqref{eq:intr_F} since (cf. \eqref{eq:phi0} and \eqref{eq:jost0})
\[
\theta_{-\frac{1}{2}}(k^2,x) -\frac{1}{\pi} \log(-k^2) \phi_{-\frac{1}{2}}(k^2,x) = \E^{\I \frac{\pi}{4}}k^{-\frac{1}{2}} f_{-\frac12}(k,x).
\]
Notice also that it suffices to consider only positive $k>0$ since $F(-k) = F(k)^\ast$ by Lemma \ref{lem:b.5}.
\end{proof}

Before proceed further, we need the following simple facts.

\begin{lemma}\label{lem:int=wro}
Suppose that $q$ satisfies \eqref{q:hyp} and \eqref{eq:q_mar2}. Then
\begin{align}
\int_0^\infty \phi_{-\frac{1}{2}}(0,s)\phi(0,s) q(s)ds &= \sqrt{\frac{\pi}{2}}\lim_{x\to \infty} W(\sqrt{x},\phi(0,x)),\label{eq:int=wro1}\\
\int_0^\infty \theta_{-\frac{1}{2}}(0,s)\phi(0,s) q(s)ds &= - 1 - \sqrt{\frac{2}{\pi}}\lim_{x\to \infty} W(\sqrt{x}\log(x),\phi(0,x)). \label{eq:int=wro2}
\end{align}
\end{lemma}

\begin{proof}
First observe that the integrals on the left-hand side are finite since 
\[
\phi_{-\frac{1}{2}}(0,x)=\sqrt{\frac{\pi x}{2}}, \quad \theta_{-\frac{1}{2}}(0,x)= -\sqrt{\frac{2x}{\pi}}\log(x),
\]
and $q$ satisfies \eqref{q:hyp} and \eqref{eq:q_mar2}. Now notice that  
\[
\int_0^x \phi_{-\frac{1}{2}}(0,s) \phi(0,s) q(s)ds = \int_0^x  \phi_{-\frac{1}{2}}(0,s) (\phi''(0,s) + \frac{1}{4s^2}\phi(0,s))ds
\]
since $\tau \phi=0$. 
Integrating by parts and noting that $\phi_{-\frac{1}{2}}(0,x)$ solves $y'' +\frac{1}{4x^2} y = 0$, we get 
\[
\int_0^x \phi_{-\frac{1}{2}}(0,s) \phi(0,s) q(s)ds =\sqrt{\frac{\pi}{2}} W(\sqrt{x},\phi(0,x))
\]
since $W(\sqrt{x},\phi(0,x)) \to 0$ as $x\to 0$. 
Passing to the limit as $x\to \infty$, we arrive at \eqref{eq:int=wro1}. The proof of \eqref{eq:int=wro2} is analogous.
\end{proof}

\begin{lemma}\label{lem:phi0atinfty}
Assume the conditions of Lemma \ref{lem:int=wro}. 
Then the equation 
\[
\tau y= -y''-\frac{1}{4x^2}y+q(x)y=0
\]
 has two linearly independent solution $y_1$ and $y_2$ such that
\be\label{eq:y1}
y_1(x) = \sqrt{x}(1+o(1)),\quad y_1'(x) = \frac{1}{2\sqrt{x}}(1+o(1))
\ee
and 
\be\label{eq:y2}
y_2(x) = \sqrt{x}\log(x)(1+o(1)),\quad y_2'(x) = \frac{\log(\sqrt{x})}{\sqrt{x}}(1+o(1))
\ee
as $x\to \infty$.
\end{lemma} 

\begin{proof}
The proof is based on successive iteration. Namely, each solution to $\tau y=0$ solves the integral equation
\[
f(x) = a\sqrt{x} + b\sqrt{x}\log(x) - \int_x^\infty \sqrt{xs}\log(x/s) f(s) q(s) ds.
\]
Since the argument is fairly standard we only provide some details for $y_2(x)$; the calculations for $y_1(x)$ are similar.
For simplicity we set $x>\E$, which is no restriction since we only need estimates for large $x$ anyway.  
As in, e.g., Lemma \ref{lem:phi} we set
\[
y_2(x) = \sum_{n=0}^\infty \phi_{n},\quad \phi_0(x):=\sqrt{x}\log(x),\quad \phi_n(x):=- \int_x^\infty \sqrt{xs}\log(x/s) \phi_{n-1}(s) q(s) ds.
\]
Since $\log(s/x) \le \log(x)\log(s)$ for all $\E\le x\le s<\infty$, we immediately get  
\[
 \abs{\phi_1(x)}\le \int_x^\infty \sqrt{xs}\log(s/x) \sqrt{s}\log(s) |q(s)| ds \le  \sqrt{x}\log(x)\int_x^\infty s \log^2(s)|q(s)| ds 
\]
and then inductively we obtain that 
\[
\abs{{\phi}_n(x)}\le \frac{\sqrt{x}\log(x)}{n!} \left( \int_x^\infty s \log^2(s)|q(s)| ds \right)^n,
\]
for all $n\in\N$ and $x\ge \E$. Therefore, we end up with the following estimate
\be
|y_2(x) - \sqrt{x}\log(x)| \le C \sqrt{x}\log(x) \int_x^\infty s \log^2(s)|q(s)| ds,\quad x\ge \E.
\ee

The derivative $y_2'(x)$ has to satisfy
\[
y_2'(x) = \frac{1}{\sqrt{x}}\left(1 + \log(\sqrt{x})\right) - \int_x^\infty \sqrt{\frac{s}{x}}\left( 1 + \log(\sqrt{x/s}) \right) y_2(s) q(s) ds.
\]
Employing the same procedure as before we set
\begin{align*}
y_2'(x) = \sum_{n=0}^\infty \beta_{n},\quad \beta_0(x)&:=\frac{1 + \log(\sqrt{x})}{\sqrt{x}},\\
\beta_n(x)&:=- \int_x^\infty \sqrt{\frac{s}{x}}\left(1 + \log(\sqrt{x/s}) \right) \beta_{n-1}(s) q(s) ds.
\end{align*}
Iteration then gives
\[
\abs{{\beta}_n(x)}\le \frac{C^{n+1}}{n!}  \frac{1+\log(\sqrt{x})}{\sqrt{x}}  \left( \int_x^\infty s \log^2(s)|q(s)| ds \right)^n
\]
for all $n\in\N$ and $x\ge \E$ since 
\[
1+\log(x/s)  \le  (1+\log(x))(1+\log(s)) \le 2\log(s)(1+\log(x)),
\]
for all $ \E\le x\le s<\infty$. Thus we end up with the estimate
\be
\left|y_2'(x) -  \frac{1+\log(\sqrt{x})}{\sqrt{x}}\right| \le C  \frac{1+\log(\sqrt{x})}{\sqrt{x}} \int_x^\infty s \log^2(s)|q(s)| ds,\quad x\ge \E,
\ee
which completes the proof.
\end{proof}

Now we are in position to characterize the behavior of $F$ near $0$. 

\begin{lemma}\label{lem:Fat0}
Suppose that $k>0$ and $q$ satisfies \eqref{q:hyp} and \eqref{eq:q_mar2}. Then 
\be\label{eq:F=F1F2}
F(k) = F_1(k) +\Big(\I- \frac{1}{\pi}\log(k^2)\Big) F_2(k), \quad k\neq 0,
\ee
where $F_1$ and $F_2$ are continuous real-valued functions on $\R$. Moreover, 
\be\label{eq:F20}
F_2(0) = \sqrt{\frac{\pi}{2}}\lim_{x\to \infty} W(\sqrt{x},\phi(0,x)) = 0
\ee
precisely when $\phi(0,x) = \OO(\sqrt{x})$ as $x\to \infty$. In the latter case 
\be\label{eq:Fresonant0}
F(k) = F_1(0)+\OO(k^2\log(-k^2)) ,\quad k\to 0,
\ee
with  
\be
F_1(0) =- \sqrt{\frac{2}{\pi}}\lim_{x\to \infty} W(\sqrt{x}\log(x),\phi(0,x)) \neq 0. 
\ee
\end{lemma}

\begin{proof}
The first claim follows from the integral representation \eqref{eq:F_rep2} since the corresponding integrals are continuos in $k$ by the dominated convergence theorem.
 Moreover, $\phi(k^2,x)$ and $\theta(k^2,x)$ are real if $k\in\R$ and hence so are $F_1$ and $F_2$. 
  
By Lemma \ref{lem:phi0atinfty}, $\phi(0,x) = ay_1(x) +by_2(x)$, where the asymptotic behavior of $y_1$ and $y_2$ is given by \eqref{eq:y1} and \eqref{eq:y2}, respectively. Combining Lemma \ref{lem:int=wro} with the representation \eqref{eq:F_rep2}, we conclude that $F_2(0)=b\sqrt{\pi/2}\neq 0$ in \eqref{eq:F=F1F2} precisely when $b\neq 0$ and hence the second claim follows. 

Assume now that $F_2(0) = 0$, which is equivalent to the equality $\phi(0,x) = ay_1(x)$ with $a= \sqrt{\pi/2}F_1(0)\neq 0$. 
Noting that both $\phi_{-\frac{1}{2}}(\cdot,x)$ and  $\phi(\cdot,x)$ are analytic for each $x>0$ and applying the dominated convergence theorem once again,
 we conclude that  
\[
\int_0^\infty \phi_{-\frac{1}{2}}(k^2,x)\phi(k^2,x)q(x) dx= \OO(k^2),\quad k\to 0.
\]
This immediately proves \eqref{eq:Fresonant0}.
\end{proof}

\begin{definition}\label{def:resonance}
We shall say that there is a resonance at $0$ if $\phi(0,x) = \OO(\sqrt{x})$ as $x\to \infty$. 
\end{definition}

Let us mention that there is a resonance at $0$ if $q\equiv 0$ since in this case $\phi(0,x) = \phi_{-\frac{1}{2}}(0,x) = \sqrt{\pi x/2}$. 

We finish this section with the following estimate. 

\begin{lemma}\label{lem:Fp}
Assume that $q$ satisfies \eqref{q:hyp} and \eqref{eq:q_mar}. Then $F$ is differentiable for all $k\neq 0$ and 
\[
\abs{F'(k)}\le \frac{C}{\abs{k}}, \quad k\neq 0.
\]
\end{lemma} 

\begin{proof}
Setting
\[
\ti{f}_{-\frac{1}{2}}(k,x) :=\frac{f_{-\frac{1}{2}}(k,x)}{f_{-\frac{1}{2}}(k)}= \E^{\I\frac{\pi}{4}}k^{-\frac{1}{2}} f_{-\frac{1}{2}}(k,x),
\]
we find that its derivative is given by (cf. \dlmf{10.6.3})
\[
\partial_k \ti{f}_{-\frac12}(k,x) = -\I  x\sqrt{\frac{\pi x}{2}} H^{(1)}_{1}(k x).
\]
Similar to \eqref{est:Jostsol} we obtain the estimate
\be 
\label{eq:partial-z-weyl-sol-free}
\abs{\partial_k \ti{f}_{-\frac12}(k,x)}
	\le C \frac{\sqrt{x(1+|k|x)}}{|k|}\E^{-\abs{\im\, k} x}
\ee
which holds for all $x>0$. Using \eqref{eq:intr_F}, we get
\[
F'(k) 
= \int_0^\infty \big(\partial_k \ti{f}_{-\frac12}(k,x)\phi(k^2,x) 
  + \ti{f}_{-\frac12}(k,x)\partial_k\phi(k^2,x) \big)q(x)dx .  
\]
The integral converges absolutely for all $k\neq 0$. Indeed, we have 
\be \label{intestnew}
1+ \log\Big(\frac{x}{y}\Big) \le (1+|\log(x)|)(1+|\log(y)|), \quad 0 < y \le x.
\ee
By \eqref{eq:partial-z-diff-phi}, \eqref{est:Jostsol} and also \eqref{intestnew}, we obtain
\begin{align*}
\left|\int_0^\infty \ti{f}_{-\frac12}(k,x)\partial_k\phi(k^2,x) q(x)dx \right| & \le C  \int_0^\infty  \sqrt{|k|}x \left( \frac{x}{1+|k|x} \right)^{\frac{3}{2}}(1+|\log(x)|)|q(x)|dx \\
& \le   \frac{C}{|k|} \int_0^\infty\ x(1+|\log(x)|) |q(x)|dx. 
\end{align*}
Using \eqref{estphi} and \eqref{eq:partial-z-weyl-sol-free} (again in combination with \eqref{intestnew}), 
we get the following estimates for the first summand:
\begin{align*}
\left|\int_0^\infty\ \partial_k \ti{f}_{-\frac12}(k,x) \phi(k^2,x) q(x)dx\right|
 \le \frac{C}{|k|} \int_0^\infty x(1+|\log(x)|)|q(x)|dx. 
\end{align*}
Now the claim follows.
\end{proof}

%%%%%%%%%%%%%%%%%%%%%%%%%%%%%%%%%%%%%%%%%%%%%%%%%%%%%%%%%%%%%%%%%%%%%%%%%
\section{Dispersive decay}
\label{ll-sec}
%%%%%%%%%%%%%%%%%%%%%%%%%%%%%%%%%%%%%%%%%%%%%%%%%%%%%%%%%%%%%%%%%%%%%%%%%

In this section we prove the dispersive decay estimate \eqref{full} for 
the Schr\"odinger equation \eqref{Schr}. In order to do this, we divide
the analysis into a low and high energy regimes. In the analysis of both regimes we make use of variants of the 
van der Corput lemma (see Appendix \ref{sec:vdCorput}), combined with a Born series approach
for the high energy regime suggested in \cite{GS} and adapted to our setting in \cite{ktt}. 

\subsection{The low energy part}
\label{ll-sec-low}
 For the low energy
regime, it is convenient to use the following well-known
representation of the integral kernel of $\E^{-\I tH}P_{c}(H)$,
\begin{align}
[\E^{-\I tH}P_{c}(H)](x,y)
	&= \frac{2}{\pi} \int_{-\infty}^{\infty}
       \E^{-\I t k^2}\,\phi(k^2,x) \phi(k^2,y) \im m(k^2) k \,dk\nn
       \\
	&=\frac{2}{\pi} \int_{-\infty}^{\infty}
	  \E^{-\I t k^2}\,\frac{\phi(k^2,x) \phi(k^2,y)k^2}{|f(k)|^2}\,dk\label{integr}
	  \\
	&=\frac{2}{\pi} \int_{-\infty}^{\infty}
	  \E^{-\I t k^2}\,\frac{\tilde{\phi}(k,x)\tilde{\phi}(k,y)}
	  {|F(k)|^2} \,dk,\nn
\end{align}
where the integral is to be understood as an improper integral. In fact, adding an additional
energy cut-off (which is all we will need below) the formula is immediate from the spectral
transformation \cite[\S3]{kst2} and the general case can then be established taking limits
(see \cite{ktt} for further details).

In the last equality we have used  
\begin{equation}\label{eq:tiphif}
\tilde{\phi}(k,x)
	:= |k|^{\frac12} \phi(k^2,x),\quad k\in \R. 
\end{equation}
Note that 
\begin{align}\label{eq:esttiphi}
|\ti{\phi}(k,x)| &\le C \left(\frac{|k|x}{1+|k|x}\right)^{\frac12}\E^{|\im k|x}\left(1+\int_0^x \left(1+ \log\Big(\frac{x}{y}\Big)\right)\frac{y|q(y)|}{1+|k|y} dy\right),\\
\label{eq:esttiphik}
|\partial_k \ti{\phi}(k,x)| &\le Cx \left(\frac{|k|x}{1+|k|x}\right)^{-\frac12}\E^{|\im k|x}\left(1+\int_0^x  \left(1+ \log\Big(\frac{x}{y}\Big)\right)\frac{y|q(y)|}{1+|k|y}dy\right),
\end{align}
which follow  from \eqref{estphil}, \eqref{estphi} and the equality
\[
\partial_k \ti{\phi}(k,x) = \frac12{\rm sgn}(k)|k|^{-\frac12} \phi(k^2,x) + |k|^{\frac12}\partial_k \phi(k^2,x)
\]
together with \eqref{eq:partial-z-phil}, \eqref{eq:partial-z-diff-phi}.

We begin with the following estimate.

\begin{theorem}\label{thm:le1}
Assume \eqref{q:hyp} and \eqref{eq:q_mar2}. Let $\chi\in C^\infty_c(\R)$ with ${\rm supp}(\chi)\subset (-k_0,k_0)$. Then
\be\label{eq:low-energy-1}
\big|[\E^{-\I tH}\chi(H)P_{c}(H) ](x,y)\big| \le {C\sqrt{xy}}{|t|^{-\frac{1}{2}}} 
\ee
for all $x,y\le 1$.
\end{theorem}

\begin{proof}
We want to apply the van der Corput Lemma \ref{lem:vC2} to the integral 
\[
I(t,x,y) := [\E^{-\I tH}\chi(H)P_{c}(H)](x,y) =\frac{2}{\pi} \int_{-\infty}^{\infty}
	  \E^{-\I t k^2}\chi(k^2)\,\frac{\tilde{\phi}(k,x)\tilde{\phi}(k,y)}
	  {|F(k)|^2} \,dk.
\]
Denote
\[
A(k) = \chi(k^2) A_0(k), \qquad A_0(k) 
=\frac{\tilde{\phi}(k,x)\tilde{\phi}(k,y)}{|F(k)|^2}.
\]
Note that 
\[
\|A\|_\infty \le \|\chi\|_\infty\|A_0\|_\infty,\qquad \|A'\|_1 \le \|\chi'\|_1\|A_0\|_\infty + \|\chi\|_1\|A_0'\|_\infty.
\]
By Lemma \ref{lem:Fprop}, $F(k)\ne 0$ for all $k\in\R\setminus\{0\}$. Moreover, combining \eqref{eq:Fatinfty} with Lemma \ref{lem:Fat0}, we conclude that $\|1/F\|_\infty<\infty$. Using \eqref{eq:esttiphi} and noting that $\log(x/y)\le \log(1/y)$ for all $0<y\le x\le 1$, we get
\be\label{eq:esttiphi<1}
|\ti{\phi}(k,x)| \le C \left(\frac{|k|x}{1+|k|x}\right)^{\frac12}\E^{|\im k|x},\quad x\in(0,1].
\ee
Therefore, 
\be\label{eq:estAinfty}
\sup_{k\in[- k_0,k_0]} |A_0(k)| \le C\,\|1/F\|^2_\infty |k_0|\sqrt{xy}, 
\ee
which holds for all $x,y\in (0,1]$ with some uniform constant $C>0$. 

Next, we get
\[
A_0'(k) = \frac{\partial_k\tilde{\phi}(k,x)\tilde{\phi}(k,y) + \tilde{\phi}(k,x)\partial_k\tilde{\phi}(k,y)}
		{\abs{F(k)}^2}
		-
		 A_0(k) \re\frac{F'(k)}{F(k)}.
\]
To consider the second term, we infer from~ \eqref{eq:esttiphi<1}, Lemma~\ref{lem:Fat0} and Lemma~\ref{lem:Fp} that 
\[
 \abs{A_0(k) \re\frac{F'(k)}{F(k)} }\le \frac{ \abss{\tilde{\phi}(k,x) \tilde{\phi}(k,y)}  }{|F(k)|^2} \abs{\frac{F'(k)}{F(k)}} \le C \sqrt{xy}.
\]
The estimate for the first term follows from \eqref{eq:esttiphi<1} and \eqref{eq:esttiphik} since 
\begin{align*}
&\left|\partial_k\tilde{\phi}(k,x)\tilde{\phi}(k,y) + \tilde{\phi}(k,x)\partial_k\tilde{\phi}(k,y)\right| \\
&\quad\quad \le C \left(\frac{|k|x}{1+|k|x}\right)^{\frac12} \left(\frac{|k|y}{1+|k|y}\right)^{\frac12} \left( \frac{1+|k|x}{|k|} + \frac{1+|k|y}{|k|}\right)\\
&\quad\quad \le C \sqrt{xy}  \frac{1+|k|x + 1+ |k|y}{\sqrt{(1+|k|x)(1+|k|y)}} \le 2C(1+|k|) \sqrt{xy},\quad x,y\in (0,1].
\end{align*}
The claim now follows by applying the classical van der Corput Lemma (see \cite[page 334]{St}) or by noting that $A\in \cW_0(\R)$ in view of Lemma \ref{lem:a.3} and then it remains to apply Lemma \ref{lem:vC2}. 
\end{proof}

\begin{theorem}\label{thm:le2}
Assume
\be
\int_0^1 |q(x)| dx <\infty \quad \text{and} \quad  \int_1^\infty x\log^2(1+x) |q(x)| dx <\infty.
\ee
Let also $\chi\in C^\infty_c(\R)$ with ${\rm supp}(\chi)\subset (-k_0,k_0)$. 
If $\phi(0,x)/\sqrt{x}$ is unbounded near  $\infty$, then
\be\label{eq:3.11}
\big|[\E^{-\I tH}\chi(H)P_{c}(H) ](x,y)\big| \le {C}{|t|^{-\frac{1}{2}}},
\ee
whenever $\max(x,y)\ge 1$.
\end{theorem}

\begin{proof}
Assume that $ 0<x\le 1\le y$. 
We proceed as in the previous proof but use Lemma~\ref{lem:toGL} and Lemma~\ref{lem:to} to write
\[
A(k) = \chi(k^2) \frac{(I+B_x) \ti{\phi}_{-\frac12}(k,x) \cdot (I+K_y) \ti{\phi}_{-\frac12}(k,y)}{|F(k)|^2},\quad k\neq 0.
\]
Indeed, for all $k\in \R\setminus\{0\}$,  $\phi(k^2,\cdot)$ admits the representation \eqref{eq:phif}. Therefore, by Lemma~\ref{lem:to}, $\ti{\phi}(k,y) = (I+K_y)\ti{\phi}_{-\frac12}(k,y)$ for all $k\in \R\setminus\{0\}$. 
 
By symmetry $A(k)=A(-k)$ and hence our integral reads
\[
I(t,x,y)=\frac{4}{\pi} \int_0^\infty \E^{-\I t k^2} A(k) dk.
\]
Let us show that the individual parts of $A(k)$ coincide with a function
which is the Fourier transform of a finite measure. Clearly, we can
redefine $A(k)$ for $k<0$. To this end note that $\ti{\phi}_{-\frac12}(k^2,x)= J(|k| x)$, where
$J(r) = \sqrt{r} J_{0}(r)$. Note that $J(r)\sim \sqrt{r}$ as $r\to 0$ and $J(r)= \sqrt{\frac{2}{\pi}} \cos(r-\frac{\pi }{4}) +O(r^{-1})$
as $r\to+\infty$ (see \eqref{eq:asymp-J-infty}). Moreover, $J'(r)\sim \frac{1}{2\sqrt{r}}$ as $r\to 0$ and $J'(r)= \sqrt{\frac{2}{\pi}} \cos(r+\frac{\pi }{4}) +O(r^{-1})$ as $r\to+\infty$ (see \eqref{eq:a11}).  Moreover, we can define $J(r)$ for $r<0$ such that it is locally in $H^1$ and $J(r)=\sqrt{\frac{2}{\pi}} \cos(r-\frac{\pi }{4})$
for $r<-1$. By construction we then have $\ti{J}\in L^2(\R)$ and $\ti{J}\in L^p(\R)$ for all $p\in(1,2)$. By Lemma \ref{lem:a.3}, $\ti{J}\in \cW_0$ and hence $\ti{J}$ is the Fourier transform of an integrable
function. Moreover, $\cos(r-\frac{\pi }{4})$ is the Fourier transform of the sum of two Dirac delta measures and so
$J$ is the Fourier transform of a finite measure.
By scaling, the total variation of the measures corresponding to $J(k x)$ is independent of $x$.

Let us show that $\chi(k^2) |F(k)|^{-2}$ belongs to the Wiener algebra $\cW_0(\R)$. As in Lemma \ref{lem:litr}, we define the functions $f_0$ and $f_1$. Since $\phi(0,x)/\sqrt{x}$ is unbounded near  $\infty$, by Lemma \ref{lem:Fat0} we conclude that $F(k) = \log(k^2)(c+o(1))$ as $k\to 0$ with some $c\neq 0$. Hence Lemma \ref{lem:Fp} yields
\[
\left| \frac{d}{dk }\frac{1}{|F(k)|^2}\right| = \left|-\frac{1}{|F(k)|^2} 2\re\left(\frac{F'(k)}{F(k)^\ast}\right)\right| \le 2\frac{|F'(k)|}{|F(k)|^3}\le  \frac{C}{|k| |\log(k)|^3}
\]
for $k$ near zero, which implies that 
\[
f_1(k)\le C\frac{1}{k \log^3(2/k)},\quad k\in(0,1).
\]
 Therefore, we get
\[
 \int_0^1\log\big(2/k\big)f_1(k)dk\le  C\int_0^1\frac{dk}{k \log^2(2/k)} = C\int_0^{1/2} \frac{dk}{k\log^2(k)} = \frac{C}{\log 2}<\infty.
\]
Noting that the second condition in \eqref{eq:ff01} is satisfied since $\chi$ has compact support and hence so are $f_0$ and $f_1$. Therefore Lemma \ref{lem:litr} implies that $\chi(k^2) |F(k)|^{-2}$ belongs to the Wiener algebra $\cW_0(\R)$.
 
Lemma~\ref{lem:vC2} then shows
\[
|\ti{I}(t,x,y)| \le \frac{C}{\sqrt{t}}, \qquad \ti{I}(t,x,y):=\frac{4}{\pi} \int_0^\infty \E^{-\I t k^2} \chi(k^2) \frac{\ti{\phi}_{-\frac12}(k,x) \ti{\phi}_{-\frac12}(k,y)}{|F(k)|^2} dk.
\]
But by Fubini we have $I(t,x,y)= (1+B_x)(1+K_y)\ti{I}(t,x,y)$ and the claim follows since both $B\colon L^\infty((0,1))\to L^\infty((0,1))$ and $K\colon L^\infty((1,\infty)) \to L^\infty((1,\infty))$ are
bounded in view of Corollary \ref{cor:Best} and Corollary \ref{cor:Kest}, respectively.

By symmetry, we immediately obtain the same estimate if $0<y\le 1\le x $. The case $\min(x,y)\ge 1$ can be proved analogously, we only need to write
\[
A(k) = \chi(k^2) \frac{(I+K_x) \ti{\phi}_{-\frac12}(k,x) \cdot (I+K_y) \ti{\phi}_{-\frac12}(k,y)}{|F(k)|^2},\quad k\neq 0.\qedhere
\]
\end{proof}

%%%%%%%%%%%%%%%%%%%%%%%%%%%%%%%%%%%%%%%%%%%%%%%%%%%%%%%%%%%%%%%%%%%%%%%%%%%%%

\subsection{The high energy part}
For the analysis of the high energy regime we use the following 
---also well-known--- alternative representation:

\begin{align}
\E^{-\I tH}P_{c}(H)
	&=\frac 1{2\pi \I}\int_{0}^{\infty}\E^{-\I t\omega}
	  \left[\cRH(\omega+\I 0)-\cRH(\omega-\I 0)\right] d\omega\nn
	  \\
	&=\frac 1{\pi \I}\int_{-\infty}^{\infty}\E^{-\I t k^2}
	  \cRH(k^2+\I 0)\,k\,dk,\label{PP}	 
\end{align}
where $\cRH(\omega)=(H-\omega)^{-1}$ is the resolvent of the 
Schr\"odinger operator $H$ and the limit is understood in the strong 
sense (see, e.g., \cite{tschroe}). We recall that for $k\in\R\setminus\{0\}$ the Green's function is given by
\be\label{eq:GreenH}
[\cRH(k^2\pm \I 0)](x,y) = [\cRH(k^2\pm \I 0)](y,x)
	= \phi(k^2,x) \frac{f(\pm k,y)}{f(\pm k)},\quad x\le y.
\ee

Fix $k_0 > 0$ and let $\chi\colon\R\to [0,\infty)$ be a $C^\infty$ function such that
\be\label{defchi0}
\chi(k^2) = \begin{cases}
0, & \abs{k}< 2k_0,\\
1, & \abs{k}> 3k_0.
\end{cases}
\ee 
The purpose of this section is to prove the following estimate.

\begin{theorem}\label{thm:he}
Suppose $q\in L^1(\R_+)$ satisfies \eqref{eq:q_mar}. Then
\[
\big|[\E^{-\I tH}\chi(H)P_{c}(H) ](x,y)\big| \le {C}{|t|^{-\frac{1}{2}}}.
\]
\end{theorem}

Our starting point is the fact that the resolvent $\cRH$ of $H$ can be expanded into the Born series
\be\label{eq:born}
\cRH(k^2\pm \I0) = \sum_{n=0}^\infty \cR_{-\frac12}(k^2\pm \I0)(-q\,\cR_{-\frac12}(k^2\pm \I0))^n,
\ee 
where $\cR_{-\frac12}$ stands for the resolvent of the unperturbed radial Schr\"odinger
operator. To this end we begin by collecting some facts about $\cR_{-\frac12}$. Its kernel is given
\[
\cR_{-\frac12}(k^2\pm \I0,x,y) = \frac{1}{k} r_{-\frac12}(\pm k,x,y),
\]
where
\[
r_{-\frac12}(k;x,y)=r_{-\frac12}(k;y,x)= k\sqrt{xy}\,J_{0}(kx)
			 H_{0}^{(1)}(ky),\quad x\le y.
\]

\begin{lemma}\label{lem:rl}
The function $r_{-\frac12}(k,x,y)$ can be written as
\[
r_{-\frac12}(k,x,y) 
	= \chi_{(-\infty,0]}(k) 
	  \int_\R \E^{\I k p} d\rho_{x,y}(p) 
	  + \chi_{[0,\infty)}(k) \int_\R \E^{-\I k p} d\rho_{x,y}^\ast(p)
\]
with a measure whose total variation satisfies
\[
\|\rho_{x,y}\| \le C.
\]
Here $\rho^\ast$ is the complex conjugated measure.
\end{lemma}

\begin{proof}
Let $x\le y$ and $k\ge 0$. Write
\[
r_{-\frac12}(k,x,y)=  J(k x) H(k y),
\]
where
\[
J(r) =  \sqrt{r}\, J_0(r), 
\quad
H(r) =   \sqrt{r}\, H_0^{(1)}(r).
\]
We continue $J(r)$, $H(r)$ to the region $r<0$ such that they are 
continuously differentiable and satisfy 
\[
J(r)=\sqrt{\frac{2}{\pi}} \cos\left(r-\frac{\pi }{4}\right),\quad
H(r)=\sqrt{\frac{2}{\pi}} \E^{\I\left(r-\frac{\pi }{4}\right)},
\]
for $r<-1$. It's enough to show that 
\[
\ti{J}(r)=J(r)-\sqrt{\frac{2}{\pi}} \cos(r-\frac{\pi }{4})
\quad \text{and}\quad \ti{H}(r)=H(r)-\sqrt{\frac{2}{\pi}} \E^{\I(r-\frac{\pi }{4})}
\]
are elements of the Wiener Algebra $\cW_0(\R)$.
In fact, they are continuously differentiable and hence it suffices to look at their asymptotic behavior.
To do this, we need the results about Bessel and Hankel functions, collected
in Appendix \ref{sec:Bessel}.
For $r<-1$ both $\ti{J}(r)$ and $\ti{H}(r)$ are zero. 
$\ti{J}$ is integrable near $0$ and for $r>1$ it behaves like $O(r^{-1})$ and  
$O(r^{-1})$ for the derivative. So $\ti{J}$ is contained in $H^1(\R)$ and 
therefore in $\cW_0$ by Lemma \ref{lem:a.3}. As for $\ti{H}$, 
near $0$ it behaves like $\sqrt{r}\log{r}$ and hence its derivative belongs to $L^p$ for all $p\in(1,2)$ near zero. 
Since $\ti{H}(r)$ and its derivative also behave like $O(r^{-1})$ for $r>1$,
Lemma \ref{lem:a.3} applies and thus we also have $\ti{H} \in \cW_0$. 
As a consequence,  both $J$ and $H$ are Fourier transforms of finite measures. By scaling the total variation of the 
measures corresponding to $J(k x)$ and $H(k y)$, are independent of $x$ and  
$y$, respectively. This finishes the proof.
\end{proof}

Now we are in position to finish the proof of the main result.

\begin{proof}[Proof of Theorem~\ref{thm:he}]
As a consequence of Lemma~\ref{lem:rl} we note
\[
|\cR_{-\frac12}(k^2\pm \I0,x,y)| \le \frac{C}{|k|}
\]
and hence the operator $q\, \cR_{-\frac12}(k^2\pm \I0)$ is bounded on $L^1$ with
\[
\|q\,\cR_{-\frac12}(k^2\pm \I0)\|_{L^1}\le \frac{C}{ |k|} \|q\|_{L^1}.
\]
Thus we get 
\begin{align*}
\abs{\inner{\cR_{-\frac12}(k^2\pm \I0)(-q\, \cR_{-\frac12}(k^2\pm \I0))^n f}{g}}
& =  \abs{\inner{-q\, \cR_{-\frac12}(k^2\pm \I0))^n f}{\cR_{-\frac12}(k^2\mp \I0)g}}
\\
&\le \norm{(-q\, \cR_{-\frac12}(k^2\pm \I0))^n f}_{L^1}
	 \norm{\cR_{-\frac12}(k^2\mp \I0)g}_{L^\infty}
\\
&\le \frac{C^{n+1}\|q\|_{L^1}^n}{|k|^{n+1}}
	 \norm{f}_{L^1}\norm{g}_{L^1}.
\end{align*}
This estimate holds for all $L^1$ functions $f$ and $g$ and hence the 
series \eqref{eq:born} weakly converges whenever 
$\abs{k}>k_0=C(l)\|q\|_{L^1}$. Namely, for all $L^1$ functions $f$ and $g$ 
we have
\be
\label{eq:born_weak}
\inner{\cRH(k^2\pm \I0)f}{g} 
	= \sum_{n=0}^\infty \inner{\cR_{-\frac12}(k^2\pm \I0)
	  (-q\, \cR_{-\frac12}(k^2\pm \I0))^nf}{g}.
\ee 
Using the estimates \eqref{estphi}, \eqref{estpsi}, \eqref{eq:a.F}, 
and \eqref{eq:Fatinfty} for the Green's function \eqref{eq:GreenH}, 
one can see that 
\[
\cRH(k^2\pm \I0)\, g\in L^\infty
\]
whenever $g\in L^1$ and $|k|>0$. Therefore, we get
\begin{multline*}
\abs{\inner{\cRH(k^2\pm \I0)(-q\, \cR_{-\frac12}(k^2\pm \I0))^nf}{g}}
\\
\begin{aligned}
& =  \abs{\inner{(-q\, \cR_{-\frac12}(k^2\pm \I0))^nf}{\cRH(k^2\mp \I0)g}}
\\
&\le \norm{(-q\, \cR_{-\frac12}(k^2\pm \I0))^nf}_{L^1}
	 \norm{\cRH(k^2\mp \I0)g}_{L^\infty}
\\
&\le \left(\frac{C\norm{q}_{L^1}}{k}\right)^n 
	 \norm{\cRH(k^2\mp \I0)g}_{L^\infty},
\end{aligned}	 
\end{multline*}
which means that $\cRH(k^2\pm \I0)(-q\, \cR_{-\frac12}(k^2\pm \I0))^n$ 
weakly tends to 0 whenever $\abs{k}>k_0$. 

Let us consider again a function $\chi$ as in \eqref{defchi0} with $k_0=C\|q\|_1$.
Since $\E^{\I tH}\chi(H)P_c = \E^{\I tH}\chi(H)$, we get from 
\eqref{PP} 
\begin{align*}
\inner{\E^{-\I tH}\chi(H) f}{g} 
= \frac 1{\pi \I}\int_{-\infty}^\infty   
   \E^{-\I t k^2}\chi(k^2)k\inner{\cRH(k^2+\I 0)f}{g} dk.
\end{align*}
Using  \eqref{eq:born_weak} and noting that we can exchange summation and 
integration, we get
\begin{multline*}
\inner{\E^{-\I tH}\chi(H) f}{g} \\
	= \frac 1{\pi \I}\sum_{n=0}^{\infty} 
	  \int_{-\infty}^\infty\E^{-\I t k^2}\chi(k^2)k
	  \inner{\cR_{-\frac12}(k^2+ \I0)(-q\, \cR_{-\frac12}(k^2+ \I0))^nf}{g} dk.
\end{multline*}
The kernel of the operator $\cR_{-\frac12}(k^2+ \I0)(-q\, \cR_{-\frac12}(k^2+ \I0))^n$ 
is given by 
\[
\frac{1}{k^{n+1}} \int_{\R_+^n} r_{-\frac12}(k;x,y_1)\prod_{i=1}^{n} q(y_i)
	\prod_{i=1}^{n-1} r_{-\frac12}(k;y_i,y_{i+1}) r_{-\frac12}(k;y_n,y)dy_1\cdots dy_n.
\]  
Applying Fubini's theorem, we can integrate in $k$ first and hence we 
need to obtain a uniform estimate of the oscillatory integral
\[
I_n(t;u_0,\ldots,u_{n+1}) 
	= \int_{\R} \E^{-\I t k^2}\chi(k^2)
	   \left(\frac{k}{2k_0}\right)^{-n}
	   \prod_{i=0}^{n}r_{-\frac12}(k;u_i,u_{i+1})\,dk
\]
since, recalling that $k_0=C(l)\|q\|_{L^1}$, one obtains
\begin{align*}
\abs{\inner{\E^{-\I tH}\chi(H) f}{g}}
	\le \frac{1}{\pi}\sum_{n=0}^{\infty} \frac{1}{(2C)^n}
		\!\!\sup_{\{u_i\}_{i=0}^{n+1}}
		\abss{I_n(t;u_0,\ldots,u_{n+1})}
		\norm{f}_{L^1}\norm{g}_{L^1}.
\end{align*}
Consider the function $f_n(k)=\chi(k^2) \left(\frac{k}{2k_0}\right)^{-n}$. Clearly, 
$f_0$ is the Fourier transform of a measure $\nu_0$ satisfying $\|\nu_0\| \le C_1$.
For $n\ge 1$, $f_n$ belongs to  $H^1(\R)$ with $\|f_n\|_{H^1} \le \pi^{-1/2} C_1(1+n)$. Hence
by Lemma~\ref{lem:vC2} and Lemma~\ref{lem:rl} we obtain
\[
\abss{I_n(t;u_0,\ldots,u_{n+1})}\le \frac{2 C_v C_1}{\sqrt{t}} (1+n) C^{n+1}
\]
implying
\[
\abs{\inner{\E^{-\I tH}\chi(H) f}{g}}\\
	\le \frac{2 C_v C_1 \, C}{\sqrt{t}} \norm{f}_{L^1}\norm{g}_{L^1} \sum_{n=0}^{\infty} \frac{1+n}{2^n}.
\]
This proves Theorem~\ref{thm:he}.
\end{proof}

%%%%%%%%%%%%%%%%%%%%%%%%%%%%%%%%%%%%%%%%%%%%%%%%%%
\appendix
%%%%%%%%%%%%%%%%%%%%%%%%%%%%%%%%%%%%%%%%%%%%%%%%%%

%%%%%%%%%%%%%%%%%%%%%%%%%%%%%%%%%%%%%%%%%%%%%%%%%%
%%%%%%%%%%%%%%%%%%%%%%%%%%%%%%%%%%%%%%%%%%%%%%%%%%
\section{The van der Corput Lemma}\label{sec:vdCorput}
%%%%%%%%%%%%%%%%%%%%%%%%%%%%%%%%%%%%%%%%%%%%%%%%%%
%%%%%%%%%%%%%%%%%%%%%%%%%%%%%%%%%%%%%%%%%%%%%%%%%%

We will need the the following variant of the van der Corput 
lemma (see, e.g., \cite[Lemma A.2]{ktt} and \cite[page 334]{St}).

\begin{lemma}\label{lem:vC2}
Let $(a,b)\subseteq\R$ and consider the oscillatory integral
\[
I(t) = \int_a^b \E^{\I t k^2} A(k) dk.
\]
If $A \in \cW(\R)$, i.e., $A$ is the Fourier transform of a signed measure
\[
A(k) = \int_\R \E^{\I k p} d\alpha(p),
\]
then the above integral exists as an improper integral and satisfies
\begin{equation*}
\abs{I(t)} 
	\le C_2 \abs{t}^{-\frac{1}{2}} \norm{A}_\cW, 
	\quad \abs{t}>0.
\end{equation*}
where $\norm{A}_\cW:=\norm{\alpha}=\abs{\alpha}(\R)$ denotes the total variation of $\alpha$ and
$C_2\le 2^{8/3}$ is a universal constant.
\end{lemma}

Note that if $A_1$, $A_2\in \cW(\R)$, then 
(cf.\ p. 208 in \cite{bo})
\[
(A_1 A_2)(k)= \frac{1}{(2\pi)^2} \int_\R \E^{\I k p} d(\alpha_1*\alpha_2)(p)
\]
is associated with the convolution
\[
\alpha_1 * \alpha_2(\Omega) = \iint \id_\Omega(x+y) d\alpha_1(x) d\alpha_2(y),
\]
where $\id_\Omega$ is the indicator function of a set $\Omega$. 
Note that
\[
\|\alpha_1 * \alpha_2\| \le \|\alpha_1\| \|\alpha_2\|.
\]

Let $\cW_0(\R)$ be the Wiener algebra of functions $C(\R)$ which are Fourier transforms of $L^1$ functions,
\[
\cW_0(\R) = \Big\{f\in C(\R)\colon f(k)=\int_{\R}\E^{\I kx} g(x)dx,\ g\in L^1(\R)\Big\}.
\]
Clearly, $\cW_0(\R)\subset \cW(\R)$. Moreover, by the Riemann--Lebesgue lemma, $f\in C_0(\R)$, that is, $f(k)\to 0$ as $k\to \infty$ if $f\in\cW_0(\R)$. A comprehensive survey of necessary and sufficient conditions for $f\in C(\R)$ to be in the Wiener algebras $\cW_0(\R)$ and $\cW(\R)$ can be found in \cite{lst12}, \cite{lt11}.  We need the following statement, which extends the well-known Beurling condition (see \cite[Lemma B.3]{HKT}). 

\begin{lemma}\label{lem:a.3} 
If $f\in L^2(\R)$ is locally absolutely continuous and $f'\in L^p(\R)$ with $p\in (1,2]$, then $f$ is in the Wiener algebra $\cW_0(\R)$ and
\be\label{eq:embedding}
\|f\|_{\cW} \le C_p\big( \|f\|_{L^2(\R)} + \|f'\|_{L^p(\R)}\big),
\ee
where $C_p>0$ is a positive constant, which depends only on $p$.
\end{lemma}

We also need the following result from \cite{lt11}.

\begin{lemma}\label{lem:litr}
Let $f\in C_0(\R)$ be locally absolutely continuous on $\R\setminus\{0\}$. Set 
\be\label{eq:f01}
f_0(x):= \sup_{|y|\ge |x|} |f(y)|, \quad f_1(x):={\rm ess\, sup}_{|y|\ge |x|} |f'(y)|,
\ee
for all $x\neq 0$. If 
\be\label{eq:ff01}
\int_0^1 \log\big(2/x\big) f_1(x)dx <\infty,\quad  \int_1^\infty \left(\int_x^\infty f_0(y) f_1(y) dy \right)^{1/2} dx<\infty,
\ee
then $f\in \cW_0(\R)$.
\end{lemma}

%%%%%%%%%%%%%%%%%%%%%%%%%%%%%%%%%%%%%%%%%%%%%%%%%%
%%%%%%%%%%%%%%%%%%%%%%%%%%%%%%%%%%%%%%%%%%%%%%%%%%
\section{Bessel functions}\label{sec:Bessel}
%%%%%%%%%%%%%%%%%%%%%%%%%%%%%%%%%%%%%%%%%%%%%%%%%%
%%%%%%%%%%%%%%%%%%%%%%%%%%%%%%%%%%%%%%%%%%%%%%%%%%

Here we collect basic formulas and information on Bessel and Hankel functions (see, e.g., \cite{dlmf, Wat}).
First of all assume $m\in \N_0$. We start with the definitions:
\begin{align}
J_{m}(z)
&=\left( \frac{z}{2} \right)^m \sum_{n=0}^\infty \frac{(\frac{-z^2}{4})^n}{n!(n+m+1)!},\label{eq:Jnu01}
	\\
Y_{m}(z)
&= -\frac{\left( \frac{-z}{2} \right)^{-m}}{\pi} \sum_{n=0}^{m-1} \frac{(m-n-1)!(\frac{z^2}{4})^n}{n!}  +\frac{2}{\pi} \log (z/2) J_m(z) \nonumber \\
&\quad +\frac{\left( \frac{z}{2} \right)^{m}}{\pi} \sum_{n=0}^{\infty} \left(\psi(n+1) + \psi(n+m+1) \right)\frac{(\frac{-z^2}{4})^n}{n!(n+m+1)!},\label{eq:Ynu01}
	\\
H_{m}^{(1)}(z) 
&= J_{m}(z) + \I Y_{m}(z), \qquad
H_{m}^{(2)}(z) 
= J_{m}(z) - \I Y_{m}(z).
\end{align}
Here $\psi$ is the digamma function \dlmf{5.2.2}. 
The asymptotic behavior as $|z|\to\infty$ is given by 
\begin{align}
J_{m}(z) 
	&= \sqrt{\frac{2}{\pi z}}\left(\cos(z- \pi m/2- \pi/4) 
	   + \E^{|\im z|}\OO(|z|^{-1})\right),
	   \quad |\arg z|<\pi\label{eq:asymp-J-infty},
	\\
Y_{m}(z) 
	&= \sqrt{\frac{2}{\pi z}}\left(\sin(z- \pi m/2- \pi/4) 
	   + \E^{|\im z|}\OO(|z|^{-1})\right),
	   \quad \abss{\arg z}<\pi\label{eq:asymp-Y-infty},
	 \\
H_{m}^{(1)}(z)
	&= \sqrt{\frac{2}{\pi z}}\E^{i(z -  \frac{2m+1}{4}\pi)}\left(1 + \OO(|z|^{-1})\right),
	   \quad -\pi<\arg z<2\pi,\label{eq:asymp-H1-infty}
	   \\
H_{m}^{(2)}(z)
	&= \sqrt{\frac{2}{\pi z}}\E^{-i(z - \frac{2m+1}{4}\pi)}\left(1 + \OO(|z|^{-1})\right),
	   \quad -2\pi<\arg z<\pi.\label{eq:asymp-H2-infty}	    
\end{align}
Using \dlmf{10.6.2}, 
one can show that the derivative of the reminder satisfies
\be\label{eq:a11}
\left(\sqrt{\frac{\pi z}{2}}J_{0}(z) - \cos(z - \pi/4) \right)'= \E^{|\im z|}\OO(|z|^{-1}),
\ee
as $|z|\to \infty$. The same is true for $Y_m$, $H_m^{(1)}$ and $H_m^{(2)}$.

\bigskip
\noindent
{\bf Acknowledgments.} We thank Vladislav Kravchenko and Sergii Torba for providing us with the paper \cite{soh}. We are also grateful to Iryna Egorova for the copy of A.\ S.\ Sohin's PhD thesis.

%%%%%%%%%%%%%%%%%%%%%%%%%%%%%%%%%%%%%%%%%%%%%%%%%%%%%%%%%%%%%%%%%%%%%

\end{document}